\documentclass[reqno,11pt]{amsart}
\usepackage{a4wide,color,eucal,enumerate,mathrsfs}
\usepackage[normalem]{ulem}
\usepackage{amsmath,amssymb,epsfig,amsthm} %,bbm
\usepackage[latin1]{inputenc}
\usepackage{psfrag}
\numberwithin{equation}{section}

\newtheorem{theorem}{Theorem}[section]

\newtheorem{lemma}[theorem]{Lemma}
\newtheorem{proposition}[theorem]{Proposition}

\newtheorem{claim}[theorem]{Claim}
\newtheorem{definition}[theorem]{Definition}

\newtheorem{example}[theorem]{Example}

\theoremstyle{remark}
\newtheorem{remark}[theorem]{Remark}
\newtheorem*{remark*}{Remark}

\DeclareMathOperator{\diam}{diam}
\DeclareMathOperator{\dist}{dist}

\newcommand{\N}{\mathbb{N}}

\newcommand{\R}{\mathbb{R}}
\newcommand{\Z}{\mathbb{Z}}
\newcommand{\mm}{\mathbf{m}}

\renewcommand{\d}{{\mathrm d}}

\def\dist{{\mathop\mathrm{\,dist\,}}}

\def\bint{{\ifinner\rlap{\bf\kern.35em--}
\int\else\rlap{\bf\kern.45em--}\int\fi}\ignorespaces}

\def\bbint{{\ifinner\rlap{\bf\kern.35em--}
\hspace{0.078cm}\int\else\rlap{\bf\kern.45em--}\int\fi}\ignorespaces}

\def\diam{{\mathop\mathrm{\,diam\,}}}

\def\dfrac{\displaystyle\frac}

\def\bint{{\ifinner\rlap{\bf\kern.35em--}
\int\else\rlap{\bf\kern.45em--}\int\fi}\ignorespaces}

\begin{document}

\title[]
{Two-sided boundary points of Sobolev-extension domains on Euclidean spaces}

\author{Miguel Garc\'ia-Bravo}
\author{Tapio Rajala}
\author{Jyrki Takanen}

\address{University of Jyvaskyla \\
         Department of Mathematics and Statistics \\
         P.O. Box 35 (MaD) \\
         FI-40014 University of Jyvaskyla \\
         Finland}
         
\email{miguel.m.garcia-bravo@jyu.fi}         
\email{tapio.m.rajala@jyu.fi}
\email{jyrki.j.takanen@jyu.fi}

\thanks{The authors acknowledge the support from the Academy of Finland, grant no. 314789.}
\subjclass[2000]{Primary 30L99. Secondary 46E35, 26B30.}
\keywords{}
\date{\today}

%%%%%%%%%%%%%%%%%%%%%%%%%%%%%%%%%%%%%%%%%%%%%%%%%%%%%%%%%%%%%%%%%%%%%

\maketitle

\begin{abstract}
We prove an estimate on the Hausdorff-dimension of the set of two-sided boundary points of general Sobolev-extension domains on Euclidean spaces. We also present examples showing lower bounds on possible dimension estimates of this type.
\end{abstract}

\section{Introduction}

We continue the investigation of the geometric properties of Sobolev-extension domains.
In this paper, the space of Sobolev functions we use on a domain $\Omega\subset\R^n$ is the homogeneous Sobolev space $L^{1,p}(\Omega)$, which is the space of locally integrable functions whose weak derivatives belong to $L^p(\Omega)$. We endow this space with the homogeneous seminorm
$$ \Vert f\Vert_{L^{1,p}(\Omega)}= \Vert \nabla f\Vert_{L^p(\Omega)}=\left( \int_{\Omega} \vert \nabla f(x)\vert^p\,dx \right)^{1/p}.$$
The reason for working with the homogeneous Sobolev space is simply to make our dimension estimates scaling invariant. We will comment on the non-homogeneous spaces after stating our main result.

We say that $E\colon L^{1,p}(\Omega)\to L^{1,p}(\mathbb{R}^n)$ is an extension operator if
there exists a constant $C \geq 1$ so that for every $u \in L^{1,p}(\Omega)$ we have
$||Eu||_{L^{1,p}(\mathbb{R}^n)} \leq C||u||_{L^{1,p}(\Omega) }$. We name the infimum over such possible constants $C$ by $||E||$, and call it the norm of the extension operator. We say that a domain $\Omega\subset\R^n$ is an $L^{1,p}$-extension domain if such an operator exists. The same definition applies for the non-homogeneous spaces $W^{1,p}(\Omega)$.

 Throughout this manuscript each time we refer to a Sobolev-extension domain we mean it with respect to the homogeneous norm, unless otherwise stated.

Already from the work of Calder\'on and Stein \cite{stein} we know that Lipschitz domains are $W^{1,p}$-extension domains. However, much more complicated domains admit an extension operator. For instance, the Koch-snowflake domains are extension domains and in some sense serve as sharp examples of extension domains in terms of the Hausdorff dimension of the boundary, see \cite{LRT2020}. In \cite{LRT2020}, the question of the possible size of the boundary for simply-connected planar Sobolev-extension domains was studied. In particular, for these domains there is an upper bound on the Hausdorff-dimension of the boundary in terms of the norm of the extension operator (although in \cite{LRT2020} the bound was expressed in terms of a constant in a characterizing curve-condition property provided in \cite{KRZ2015}).
Note that for the boundary of a general extension domain we cannot have a dimension estimate: take $\Omega = [0,1]^n \setminus C^n$ with $C$ a Cantor set of $\dim_\mathcal{H}(C) = 1$ but Lebesgue measure zero. Then $\Omega$ is a Sobolev $L^{1,p}$-extension domain, but $\dim_\mathcal{H}(\partial \Omega) = n$.

With a bound on the dimension of the boundary, one might wonder what other geometric limitations does the existence of an extension operator imply. A basic example of a domain that is not an $L^{1,p}$-extension domain for any $p$ is the slit disc: $\Omega := \mathbb D \setminus \{0\}\times[0,1] \subset \mathbb R^2$. A continuous Sobolev function in $L^{1,p}(\Omega)$ which is one above, say on $\{0\}\times[1/2,1]$ and zero below it serves as an example of a function that cannot be extended to a global Sobolev function. One reason for not having an extension is that there is a set of positive $1$-dimensional Hausdorff measure where we approach the boundary of $\Omega$ from two different sides. In particular, no extension would be absolutely continuous on almost every vertical line segment. However, the slit disc is an example of a $BV$-extension domain because its complement is quasiconvex (see \cite{KMS2010}).

The slit disc example can be modified to a more delicate one by replacing the removed line segment $\{0\}\times[0,1]$ by a larger set where the two-sided points are at a Cantor set on the previously removed line segment. This will give a domain where the extendability of Sobolev functions depends on the exponent $p$. Such constructions will also play a role in this paper, see Section \ref{sec:sharpness} (also for the precise definitions of these domains). By removing small neighbourhoods of the two-sided points from the domain, one can actually make the example into a Jordan domain and still retain the critical extendability properties, see
\cite{KYZ2010} (and also the earlier works \cite{mazya1981,Romanov}).

In this paper we study the question of how large the set of two-sided points can be for a Sobolev-extension domain. This question was already investigated in \cite{Jyrki} by the third named author in the case of planar simply-connected domains. Before continuing, let us give the definition of two-sidedness that we will use in this paper. In the case of simply-connected planar domains, the definition can also be reformulated in various ways using conformal mappings, see \cite{Jyrki}.

\begin{definition}[Two-sided points of the boundary of a domain]
 Let $\Omega \subset \R^n$ be a domain. A point $x \in \partial\Omega$ is called \emph{two-sided}, if there exists $R>0$ such that for all $r \in (0,R)$ there exist disjoint connected components $\Omega_r^1$ and $\Omega_r^2$ of $\Omega \cap B(x,r)$ that are nested: $\Omega_s^i \subset \Omega_r^i$ for $0 < s < r < R$ and $i \in \{1,2\}$.
\end{definition}

 We denote the set of two-sided points of $\partial\Omega$ by $\mathcal{T}_\Omega$, or simply by $\mathcal T$, if there is no possibility for confusion.
Notice that the set $\mathcal T$ need not be closed. 

For $p\geq n$, we know that $L^{1,p}$ extension domains are quasiconvex (see \cite[Theorem 3.1]{K1990}). Therefore, for an $L^{1,p}$-extension domain with $p \ge n$, we have $\mathcal{T} = \emptyset$. The interesting case is thus $1 \le p  < n$.
For this range we prove the following estimate on the size of $\mathcal T$:

\begin{theorem}\label{thm:main}
Let $n \ge 2$ and $p \in [1,n)$  and let $\Omega \subset \R^n$ be a Sobolev $L^{1,p}$-extension domain.   Then
\begin{enumerate}
 \item If $p=1$, then $\mathcal H^{n-1}(\mathcal{T}_\Omega) = 0$.
    \item If $p>1$ there exists a constant $C(n,p)>0$ so that
\[
 \dim_{\mathcal H}(\mathcal{T}_\Omega) \le n-p - \frac{C(n,p)}{\|E\|^{n}\log(\|E\|)},
\]
where $\|E\|$ is the operator norm of the homogeneous Sobolev-extension operator.
\end{enumerate}
 Here we use the convention that $\mathcal{T}_{\Omega}=\emptyset$ whenever the bound on the right-hand side of the estimate is strictly less than $0$. 
\end{theorem}

Let us now comment on the non-homogeneous Sobolev spaces.
For bounded domains $\Omega$ it is known that Sobolev $L^{1,p}$-extension domains are the same as Sobolev $W^{1,p}$-extension domains (see \cite{K1990}), so even though our main result is stated for homogeneous Sobolev-extension domains, it can be applied to $W^{1,p}$-extension domains in the case that $\Omega$ is bounded. Let us note that there exist unbounded Sobolev $W^{1,p}$-extension domains which are not $L^{1,p}$-extension domains (see \cite[Example 6.7]{K1990}). However, one might expect that our result still applies for this unbounded case because having a dimension bound relies on local properties. Indeed, our method of proof will show that we can handle also with unbounded $W^{1,p}$-extension domains because the measure density condition is still true for every $r\in (0,1)$ and the proof of Theorem \ref{thm:main} studies locally the set of two-sided points to estimate its dimension. We prefer to state our main theorem only for $L^{1,p}$-extension domains because of their homogeneous norm. If we stated it for $W^{1,p}$-extension domains then a scaling of the domain $\Omega$ would perturb the norm of the operator $E$, and hence our estimate in the dimension of the two-sided points. Obviously, a scaling of a set will never change its dimension.

We will also give a size estimate on the two sided-points of $BV$-extension domains.
\begin{theorem}\label{thm:BV}
Let  $n \ge 2$ and let $\Omega \subset \R^n$ be a $BV$-extension domain. Then
\[
\dim_{\mathcal H}(\mathcal{T}_\Omega) \le n-1.
\]
\end{theorem}

Observe that taking $\Omega$ to be a slit disc shows the  sharpness of this result.

We will present the proof of Theorem \ref{thm:main} in Section \ref{sec:proof} and the proof of Theorem \ref{thm:BV} in Section \ref{sec_BV}. After that, in Section \ref{sec:examples}, we show that Theorem \ref{thm:main} (1) is sharp: there exist even a planar simply connected $L^{1,1}$-extension domain with $\dim_{\mathcal H}(\mathcal{T}) = 1$. We also give a class of domains $\Omega_\lambda$ for each $n \ge 2$ with the sets of two-sided points $\mathcal T_\lambda = C_\lambda$ being Cantor-sets, so that for every $p$ there exists a constant $C(n,p)$ for which, with the explicit 
extension operators $E_\lambda \colon L^{1,p}(\Omega_\lambda) \to L^{1,p}(\mathbb R^n)$ we construct, we have $\|E_\lambda\| \to \infty$ as $\dim_{\mathcal H}(C_\lambda)\to n-p$ and the estimate
\begin{equation}\label{eq:example}
\dim_{\mathcal H} (C_{\lambda})\geq n-p-\frac{C(n.p)}{\| E_\lambda\|}
\end{equation}
is satisfied. 

This set of examples together with Theorem \ref{thm:main} shows that the possible optimal asymptotic behaviour for the dimension bound of the two sided points in terms of the norm of the extension operator is between $n-p - C/\|E\|$ and $n-p -C/(\|E\|^{n}\log(\|E\|))$.
We note that in \cite{Jyrki} the exponents for the dimension bound and examples agreed, thus providing a possibly sharper estimate. However, as the study in \cite{Jyrki} was done in terms of a constant in a characterizing curve condition, and since the dependence between this constant and the norm of the extension operator has not been clarified, the estimate in \cite{Jyrki} does not yet translate to a sharp dimension estimate in terms of the norm of the extension operator in the planar simply connected case.

\section{Dimension estimate for the set of two-sided points}\label{sec:proof}

In this section we will prove Theorem \ref{thm:main}. Before doing so, we go through some notation and lemmata.

We often denote by $C(\cdot)$ a computable constant depending only on the parameters listed in the parenthesis. The constant may differ between appearances, even within a chain of (in)equalities. By $a\lesssim b$ we mean that $a \leq Cb$ for some
constant $C \geq 1$, that could depend on the dimension $n$. Similarly for $a\gtrsim b$. Then $a \sim b$ means that both $a\lesssim b$ and $b \gtrsim a$ hold. 
We denote by $\mm_n$ the $n$-dimensional Lebesgue measure on $\R^n$.

We will also denote by $Q(x,s)$ the cube of center $x$ and side length $s>0$ and for a given cube $Q=Q(x,s)$ and some positive $K>0$ we write $KQ=Q(x,Ks)$.

We will use the following basic lemma, similar to \cite[Lemma 3.2]{Jyrki}.
\begin{lemma}\label{lma:dim}
Let $F \subset \mathbb R^n$, $0 < \lambda < 1$, and $s \ge 0$. For every $i \in \mathbb N$ let $\{x_k^i\}_{k\in I_i}$ be a maximal $\lambda^i$-separated net in $F$.
Assume that for each $i \in \mathbb N$ and $k \in I_i$ there exists $j > i$ such that
$$N_j < \lambda^{-(j-i) s },$$
where $N_j = \# (\{ l \in I_j : B(x^{j}_{l}, \lambda^{j} )\cap B(x^{i}_{k}, \lambda^{i}) \neq \emptyset \})$. Then $\dim_{\mathcal H } (F) \leq s$.
\end{lemma}

A measure density condition for Sobolev-extension domains was proven in \cite{HKT2008}. We will need to make the dependence of the parameters more explicit, so we modify slightly the proofs of \cite[Lemma 11]{HKT2008} and  \cite[Theorem 1]{HKT2008} to obtain the following version of their measure density condition.

\begin{proposition}[Measure density condition]\label{prop:md}
Let $1\leq p<\infty$ and let  $\Omega \subset \R^n$ be a Sobolev $L^{1,p}$-extension domain with an extension operator $E$. 

\begin{enumerate}
    \item If $1\leq p < n$ then for all $x \in \overline\Omega$ and $r\in \left(0,\min \left\lbrace 1, \left(\frac{\mm_n(\Omega)}{2\, \mm_n(B(0,1))}\right)^{1/n}\right\rbrace \right)$, denoting by $\Omega'$ a connected component of $ \Omega \cap B(x,r)$  with $x\in \overline{ \Omega'}$, we have 
\[
 \mm_n(\Omega') \ge C(n,p)\|E\|^{-n}r^n.
\]
\item If $p> n-1$ then for for all  $x \in \overline\Omega$ and $r\in \left(0, \diam(\Omega)\right)$, denoting by $\Omega'$  a connected component of $ \Omega \cap B(x,r)$ with $x\in \overline{\Omega'}$, we have 
\[
 \mm_n(\Omega') \ge C(n,p)\|E\|^{-p}r^n.
\]
\end{enumerate}
\end{proposition}
\begin{proof}

The case (2) follows by Theorem 2.2 and Theorem 4.1 from Koskela's dissertation \cite{K1990}, where he uses the concept of variational $p$-capacity.

We look now at the case $1\leq p<n$.

Let us denote $r_0 = r$.
By induction, we define for every $i \in \N$ the radius $r_i \in (0,r_{i-1})$ by the equality
\[
\mm_n(\Omega'\cap B(x,r_i)) = \frac12\mm_n(\Omega'\cap B(x,r_{i-1})) = 2^{-i}\mm_n(\Omega').
\]
Since $x \in \Omega'$, we have that $r_i \searrow 0$ as $i\to \infty$.

For each $i \in\N$, consider the function $f_i:\Omega \to \mathbb{R}$
\[
 f_i(y) = \begin{cases}1, & \text{for }y \in B(x,r_i) \cap \Omega',\\
 \frac{r_{i-1}-|x-y|}{r_{i-1}-r_i}, & \text{for }y \in (B(x,r_{i-1})\setminus B(x,r_i)) \cap \Omega',\\
 0, & \text{otherwise}.
 \end{cases}
\]
For the homogeneous Sobolev-norm of $f_i$ we can estimate
\begin{equation}\label{eq:homonormestimate}
\begin{split}
\|f_i\|_{L^{1,p}}^p  & = \int |\nabla f_i|^p  \le |r_i - r_{i-1}|^{-p} \mm_n((B(x,r_{i-1})\setminus B(x,r_i))\cap \Omega') \\
&=|r_i - r_{i-1}|^{-p} 2^{-i}\mm_n(\Omega').
\end{split}
\end{equation}

Call $p^*=\frac{np}{n-p}$.
For any $Ef_i\in L^{1,p}(\R^n)$ we know the existence of a constant $c_i\in\R$ (that we can assume is between $0$ and $1$) so that
$$\Vert Ef_i-c_i \Vert_{L^{p*}(\R^n)}\leq C(n,p) \Vert Ef_i\Vert_{L^{1,p}(\R^n)}.$$
Hence we have the following chain of inequalities.
$$\|f_i-c_i\|_{L^{p*}(\Omega)}\leq \|Ef_i-c_i\|_{L^{p*}(\mathbb{R}^n)}\leq C(n,p)\|Ef_i\|_{L^{1,p}(\mathbb{R}^n)}\leq C(n,p)\|E\|\,\|f_i\|_{L^{1,p}(\Omega)}. $$
Recall that by our choice of $r=r_0$ we always have
$$\mm_n(\Omega\setminus B(x,r_{i-1}))\geq  \mm_n(\Omega\setminus B(x,r_0))\geq \mm_n(\Omega)- \mm_n(B(x,r_0))\geq \frac{\mm_n(\Omega)}{2}$$
and 
$$\mm_n( B(x,r_{i}) \cap  \Omega')\leq \frac{\mm_n(\Omega)}{2} $$
for every $i\geq 1$. Then
\begin{align*}
\int_{\Omega} \vert f_i(y)-c_i\vert^{p*}\,dy &\geq \max \left\lbrace\int_{\{y:\,f_i(y)=0\}}\vert c_i\vert^{p*}\,dy,\int_{\{y:\,f_i(y)=1\}}\vert 1-c_i\vert^{p*}\,dy\right\rbrace  \\
&= \max\left\lbrace  \vert c_i\vert^{p^*}  \mm_n(\Omega\setminus B(x,r_{i-1})), \vert 1-c_i\vert^{p^*} \mm_n(B(x,r_i)\cap \Omega') ) \right\rbrace \\
&  \geq \mm_n( B(x,r_{i}) \cap  \Omega')\cdot \max \left\lbrace  \vert c_i\vert^{p^*}, \vert 1-c_i\vert^{p^*} \right\rbrace    \geq   \mm_n( B(x,r_{i}) \cap  \Omega') \cdot 2^{-p*} ,
\end{align*}
so we write, using \eqref{eq:homonormestimate},
\begin{align*}
2^{-p^*-i}\mm_n(\Omega') &=2^{-p^*}\mm_n(B(x,r_i)\cap \Omega') \leq \|f_i-c_i\|^{p^*}_{L^{p^*}(\Omega)}\leq C(n,p)\|E\|^{p^*}\|f_i\|^{p^*}_{L^{1,p}(\Omega)} \\
&\leq C(n,p)\|E\|^{p^*}  \left(\int_{\Omega}|\nabla f_i(y)|^p\, dy\right)^{p^*/p}          \\
&\leq C(n,p)\|E\|^{p^*}  \left( |r_i - r_{i-1}|^{-p} 2^{-i}\mm_n(\Omega')  \right)^{p^*/p}          \\
&\leq C(n,p)\|E\|^{p^*} 2^{-ip^*/p}\mm_n(\Omega')^{p^*/p}|r_{i-1}-r_i|^{-p^*}.
\end{align*}
Consequently,
\begin{align*}
r_{i-1}-r_i & \leq C(n,p)\|E\|2^{i(1/p^{*}-1/p)}\mm_n(\Omega')^{1/p-1/p^{*}} \\
&= C(n,p)\|E\|2^{-i/n}\mm_n(\Omega')^{1/n}.
\end{align*}
By summing up all these quantities we conclude that
$$ r =r_0= \sum_{i=1}^\infty(r_{i-1} - r_{i}) \le C(n,p)\|E\|\sum_{i=1}^\infty 2^{-i/n}\mm_n(\Omega')^{1/n} = \frac{C(n,p)\|E\|}{2^{1/n}-1}\mm_n(\Omega')^{1/n}. $$
This gives the claimed inequality.
\end{proof}

Observe that the measure density condition only holds for $1\leq p<\infty$. For $W^{1,\infty}$-extension domains this is not true.  Take for instance $C\subset [0,1]$ a  fat Cantor set with $\mm_1(C)>0$. Then almost every point of $C$ is of density $1$ on $C$, so $[0,1]\setminus C$, whose closure is the whole interval $[0,1]$, cannot satisfy any measure density condition. Then take $\Omega =\mathbb{R}^n\setminus C^n$ which, by \cite[Theorem A]{HH2008} will be quasiconvex, and consequently a $W^{1,\infty}$-extension domain by \cite[Theorem 7]{HKT2008}, but does not satisfy any measure density condition either.

In the proof of Theorem \ref{thm:main} we will use the following consequence of a Sobolev-Poincar\'e type inequality \eqref{SobPonPunctCube}. The proof of the lemma follows the proof for the classical Sobolev-Poincar\'e inequality that can be found in many text books. However, for our application of the lemma we need to include a set $F$ that is removed when integrating the gradient of the Sobolev function. This fact forces us to be more cautious. For the convenience of the reader, we provide here the proof with the needed modifications.

\begin{lemma}
\label{lma:sobpoinc}
Let $1 \le p < n$, $Q \subset \mathbb R^n$ be a cube, $\delta \in (0,1)$ and $F \subset Q$ such that for any $i \in \{1, \dots, n\}$ we have 
\[
 \mm_{n-1}(P_i(F)) \le \frac{\delta}{2n\cdot 2^n}\mm_{n-1}(P_i(Q))
\]
with $P_i$ the projection $P_i \colon (x_1,\dots,x_n) \mapsto (x_1,\dots, x_{i-1},x_{i+1}, \dots, x_n)$.
Then for any $f \in W^{1,p}(Q)$ so that $0\leq f\leq 1$ and  
\[
\min\left(\mm_n(\{y \in \frac{1}{2}Q\,:\, f(y) = 0\}), \mm_n(\{y \in \dfrac{1}{2}Q\,:\, f(y) = 1\})\right) > \delta \dfrac{\ell(Q)^n}{2^n},
\]
 we have
\begin{equation}\label{eq:largeenergy}
\int_{Q\setminus F}|\nabla f(y)|^p\,\d y \ge C(n,p)\delta^{\frac{n-p}{n}} \ell(Q)^{n-p}.
\end{equation}

\end{lemma}

\emph{Remark.} 
Observe that for the conclusion of Lemma \ref{lma:sobpoinc} it is not enough to only require $\mm_n(F)$ to be small. Consider for instance the cube minus a very thin central band which separates the cube in two connected components.

\begin{proof}
We will show that the following version of Sobolev-Poincar\'e inequality holds for our function $f$:
\begin{equation}\label{SobPonPunctCube}
\left(\int_{A}|f(y)-f_A|^{\frac{pn}{n-p}}\, \d y\right)^{\frac{n-p}{pn}} \le C(n,p)\left(\int_{Q\setminus F}|\nabla f(y)|^p\,\d y \right)^{1/p},
\end{equation}
where $A = \left\{x \in \frac12Q \,:\, P_i(x) \notin P_i(F)\text{ for every }i\right\}$ and
$$f_A=\dfrac{1}{\mm_n(A)}\int_A f(y)\,dy. $$ 
Let us first observe that this implies the inequality \eqref{eq:largeenergy}.   

\begin{align*}
\int_{Q\setminus F}|\nabla f(y)|^p\,\d y &\gtrsim \left(\int_{A}|f(y)-f_A|^{\frac{pn}{n-p}}\, \d y\right)^{\frac{n-p}{n}} \\
&\gtrsim \max \left(\mm_n(\{y\in A:\,f(y)= 1\})^{\frac{n-p}{n}} |1-f_A|^p, \mm_n(\{y\in A:\,f(y)=0\})^{\frac{n-p}{n}} |f_A|^p \right) \\
& \gtrsim  \delta^{\frac{n-p}{n}}\ell(Q)^{n-p}\max  \left( |1-f_A|^p,  |f_A|^p  \right)  \\
&\gtrsim \delta^{\frac{n-p}{n}}\ell(Q)^{n-p}.
\end{align*}
Here we used the simple observation that $\mm_n(\frac{1}{2}Q\setminus A) \leq \frac{\delta}{4\cdot 2^n} \mm_n(Q)$.

To prove \eqref{SobPonPunctCube} we start by presenting the Sobolev-embedding in the form
\begin{equation}\label{soblev-poin-modif}
\left(\int_{A'}|g(y)|^{\frac{pn}{n-p}}\, \d y\right)^{\frac{n-p}{pn}} \le C(n,p,K)\left(\int_{Q\setminus F}|\nabla g(y)|^p\,\d y \right)^{1/p},
\end{equation}
for all $g \in W_0^{1,p}(Q)$  with $ |g|\leq 1$ and $\mm_n( \{x\in A': |g(x)|\geq 1/2\})\geq K\delta \ell(Q)^n$ for some positive constant $K>0$, and where $A' = \left\{x \in Q \,:\, P_i(x) \notin P_i(F)\text{ for every }i\right\}$. Following the proof of \cite[Theorem 4.8]{EG2015} what we first get is
$$\left(\int_{A'}|g(y)|^{\frac{pn}{n-p}}\, \d y\right)^{ \frac{n-1}{n}} \le C(n,p)\left(\int_{Q\setminus F}|g(y)|^{\frac{pn}{n-p}}\,\d y \right)^{\frac{p-1}{p}}\left(\int_{Q\setminus F}|\nabla g(y)|^p\,\d y \right)^{1/p}.$$
Note that by the properties of $g$ and by definition of $A'$
\[
\int_{Q\setminus A'}|g(y)|^{\frac{pn}{n-p}}\,\d y\leq \mm_n(Q\setminus A')<\dfrac{n\delta}{2n\cdot 2^n}\ell(Q)^n 
\]
and
\[
\int_{A'}|g(y)|^\frac{pn}{n-p}\,\d y\geq \left(\frac{1}{2}\right)^{\frac{pn}{n-p}}K\delta\ell(Q)^n.
\]
Therefore,

\begin{equation}\label{eq:blabla}
\begin{split}
\int_{Q\setminus F}|g(y)|^{\frac{pn}{n-p}}\,\d y \leq \int_{Q}|g(y)|^{\frac{pn}{n-p}}\,\d y & = \int_{A'}|g(y)|^{\frac{pn}{n-p}}\,\d y   + \int_{Q\setminus A'}|g(y)|^{\frac{pn}{n-p}}\,\d y\\
& \leq (1+C(n,p,K))\int_{A'}|g(y)|^{\frac{pn}{n-p}}\,\d y, 
\end{split}
\end{equation}
and finally we can get \eqref{soblev-poin-modif}.

Secondly, we apply the inequality \eqref{soblev-poin-modif}   to the function $g(y)=(f(y)-f_{A'})\phi(y)$, where $\phi\in C^{\infty}_{0}(\R^n)$ is supported in $Q$, is equal to $1$ on $\frac{1}{2}Q$ and $|\nabla \phi|\lesssim \dfrac{1}{\ell(Q)}$. We get
\begin{align}\label{sob-poin.modif}
\left(\int_{A}|f(y)-f_A|^{\frac{pn}{n-p}}\, \d y\right)^{\frac{n-p}{pn}} &\le  \left(\int_{A'}|(f(y)-f_A)\phi(y)|^{\frac{pn}{n-p}}\, \d y\right)^{\frac{n-p}{pn}} \nonumber \\
&\leq C(n,p)\left(\int_{Q\setminus F}|\nabla f(y)|^p\,\d y\right)^{1/p}\\ 
& \quad +\frac{C(n,p)}{\ell(Q)}\left(\int_{Q\setminus F}|f(y)-f_{A}|^p\,\d y\right)^{1/p}. \nonumber
\end{align}
To handle the last term above, we first prove that 
\begin{align*}
 \left(\int_{Q\setminus F}|f(y)-f_{A}|^{p}\, dy \right)^{1/p}&\le C(n,p) \left(\int_{A'}|f(y)-f_{A}|^{p}\, dy\right)^{1/p} \\
 &\leq C(n,p) \left(\left(\int_{A'}|f(y)-f_{A'}|^{p}\, dy\right)^{1/p}+\left(\int_{A'}|f_{A'}-f_{A}|^{p}\, dy\right)^{1/p}\right)\\
 &\leq C(n,p)\left(\int_{A'}|f(y)-f_{A'}|^{p}\, dy\right)^{1/p}.
\end{align*} 
In the first inequality we are using a similar trick like in \eqref{eq:blabla} (that $0\leq f\leq 1$ and that $f= 1$ and $f=0$ in  large enough sets). In the last inequality we use  H\"older inequality and the fact that  $\mm_n(A')/\mm_n(A)\leq C(n)$. 

Finally, by modifying the standard proof for the Poincar\'e inequality (see \cite[Section 4.5.2]{EG2015}) by first writing
$$
|f(y) - f(x)| \le \sum_{i=1}^n|f(z_i) - f(z_{i-1})|,
$$
with $z_i = (y_1, \dots, y_i, x_{i+1}, \dots x_n)$ so that $z_i$ and $z_{i-1}$ differ only in one coordinate, we are able to consider absolute continuity only along lines going in the coordinate directions. Thus, we obtain
$$ 
 \int_{A'}|f(y)-f_{A'}|^{p} \,dy\leq C(n,p) \ell(Q)^p\int_{Q\setminus F} |\nabla f(y)|^p\, dy.
$$ 
Combining the above with \eqref{sob-poin.modif} concludes the proof.
\end{proof}

We are now ready to prove Theorem \ref{thm:main}.

\begin{proof}[Proof of Theorem \ref{thm:main}]
Let us first make some initial reductions.
By the definition of two-sided points we can write $\mathcal{T}_{\Omega}=\bigcup_{i\in \mathbb N} \mathcal{T}_i$,  where
\begin{align*}
\mathcal{T}_i=\{x\in \partial \Omega\,: \,& \text{for every } r<2^{-i,}\text{ there exist two different connected components }\Omega^{1}_{r},\Omega^{2}_{r}\\
&\text{of } \Omega\cap B(x,r)\text{ that are nested, that is }\Omega^{j}_{s}\subset \Omega^{j}_{r}\text{ for } 0<s<r,\; j=1,2\}.
\end{align*}
Observe that if $x\in\mathcal T_i$ and  $\Omega^{1}_{r}$, $\Omega^{2}_{r}$ are the associated nested connected components of $\Omega\cap B(x,r)$ for each $r\in (0,2^{-i})$,  then $x\in \overline{\Omega^{1}_{r}}\cap \overline{\Omega^{2}_{r}}$ for all $r\in (0,2^{-i})$.

It is clear that it is enough to estimate $\dim_{\mathcal{H}}(\mathcal{T}_i)$ for a fixed $i\in\mathbb{N}$. We now cover $\mathcal{T}_i$ by countably many balls $B(z_k, 2^{-i}/6)$, where $z_k\in\mathcal{T}_i$. Then, for every $k\in\mathbb{N}$ we introduce the family of disjointed connected components of $B(z_k,2^{-i}/2)\cap \Omega$, which we denote by $\{O^{k}_{l}\}_{l\in I}$.  Let us check  now that
\begin{equation}\label{eq:formula_1}
\mathcal{T}_i\cap B(z_k,2^{-i}/6)\subseteq\bigcup_{l\neq \widetilde{l}} \overline{O^{k}_{l}}\cap \overline{O^{k}_{\widetilde{l}}}.
\end{equation}
Take $x\in \mathcal{T}_i\cap B(z_k,2^{-i}/6)$. Since $x\in\mathcal{T}_i$ there exist two different connected components of $\Omega\cap B(x,2^{-i})$, which we call $U_1,U_2$, so that $x\in \overline U_1\cap \overline U_2 $. 
Therefore, using that 
$$B(z_k,2^{-i}/2)\cap \Omega\subset B(x,2^{-i}) \cap \Omega,$$
the sets $U_1\cap B(z_k,2^{-i}/2)$ and $U_2\cap B(z_k,2^{-i}/2)$ will have connected components, which we call $O^k_l, O^k_{\tilde l}$, so that $x\in \overline {O^k_l}\cap \overline {O^k_{\tilde l}}$. We have then proved \eqref{eq:formula_1}.
Observe that we can write
$$\mathcal T_{\Omega}=\bigcup_{i,k,l,\tilde l} \mathcal{T}_i\cap B(z_k,2^{-i}/6)\cap(\overline{O^{k}_{l}}\cap \overline{O^{k}_{\widetilde{l}}}). $$
Therefore, it is enough to just estimate the Hausdorff dimension of 
$\mathcal{T}_i\cap B(z_k,2^{-i}/6)\cap(\overline{O^{k}_{l}}\cap \overline{O^{k}_{\widetilde{l}}}) $
for fixed $i,k,l,\widetilde{l}$.  Each set of this type, that we call from now on $G$, has the following properties: there is some $x_0\in \partial \Omega$  and some radius $r\in(0,1)$  so that 
    $$G \subset \partial \Omega\cap B(x_0,r),$$
    and there exist connected components $\Omega_1,\Omega_2 \subset \Omega\cap B(x_0,3r)$ for which 
    $$G \subset \partial \Omega_1\cap \partial\Omega_2.$$
We will now proceed to estimate the Hausdorff dimension of such a set $G$.

\medskip

(1) Let us first prove that $\mathcal{H}^{n-p}(G)=0$ for all $1\leq p<n$. In particular, this will handle the case $p=1$ in the claim (1) of the theorem. We will use the well-known fact that for any given $h\in L^{1}_{\text{loc}}(\R^n)$ and  $0 \leq s <n$ we have 
\begin{equation}\label{thm 2.10. Evans-Gariepy}
\mathcal{H}^{s}\left(\left\{x\in\mathbb{R}^n:\,\limsup_{\varepsilon \to 0}\frac{1}{\varepsilon^s}\int_{B(x,\varepsilon)}|h(y)|\,dy>0\right\}\right)=0.
\end{equation}
See for instance \cite[Theorem 2.10]{EG2015} for a proof of this assertion.

We start by defining a function $u \in L^{1,p}(\Omega)$,
\[
 u(x) = \max\left(0, \min\left(1, 3-r^{-1}\dist(x,x_0)\right)\chi_{\Omega_1}(x) \right),
\] 
where $\chi_{\Omega_1}$ denotes the indicator function of the set $\Omega_1$. By Proposition \ref{prop:md}, for every $x\in G$  and every $\varepsilon<r\leq 1$,
\[
\min\left(\mm_n(\Omega_1\cap B(x,\varepsilon/2\sqrt{n}),\mm_n(\Omega_2\cap B(x,\varepsilon/2\sqrt{n}))\right) \ge C(n,p)\|E\|^{-n}\varepsilon^{n}.
\]
Now by Lemma \ref{lma:sobpoinc}, where the removed set $F=\emptyset$, for the corresponding cube $Q(x,2\varepsilon/\sqrt{n})$ centered at $x$ and with side length $2\varepsilon/\sqrt{n}$ (thus containing the ball $B(x,\varepsilon/\sqrt{n})$ and contained in the ball $B(x,\varepsilon)$), we have
\[
\int_{B(x,\varepsilon)}|\nabla Eu(y)|^p\,\d y \ge 
\int_{Q(x, 2\varepsilon/\sqrt{n})}|\nabla Eu(y)|^p\,\d y \ge  C(n,p)\|E\|^{p-n}\varepsilon^{n-p}.
\]
Therefore,
$$\limsup_{\varepsilon\to 0}\frac{1}{\varepsilon^{n-p}}\int_{B(x,\varepsilon)}|\nabla Eu(y)|^p\,dy\geq C(n,p) \|E\|^{p-n} >0 $$
for every $x\in G$, and using \eqref{thm 2.10. Evans-Gariepy} we conclude $\mathcal{H}^{n-p}(G)=0$.

We are done with the case $p=1$ of Theorem \ref{thm:main}. For the case $p>1$ we will be able to be more precise in the estimation of the Hausdorff dimension in terms of the norm of the extension operator $E\colon L^{1,p}(\Omega)\to L^{1,p}(\R^n)$. For this we will follow a different approach.
\medskip

(2) Let us now focus on the case $p>1$.
First of all, call $C_1(n,p)$ and $C_2(n,p)$ the constants given by Proposition \ref{prop:md} and Lemma \ref{lma:sobpoinc} respectively. Now we choose  $0<\lambda<r$ small enough so that
$$\dfrac{\lambda^{p-1}}{1-\lambda^{p-1}}\leq  \dfrac{C_1(n,p)^{1+\frac{n-p}{n}}C_2(n,p)}{2^{2n+1}3^n  n^{(2n+1-p)/2} \mm_n(B(0,1))} \|E\|^{-2n}.$$ 
We can do this because the term on the left hand side tends to zero as $\lambda\to 0$.

For every $i \in \mathbb N$, let $\{x_k^i\}_{k \in I_i}$ be a maximal $2\lambda^i$-separated net of points in $G$. For every $i \in \mathbb N$ and $k \in I_i$ define
\[
\mathcal {B}_j^{i,k}=\{ B(x_l^{i+j},\lambda^{i+j} ): B(x_l^{i+j}, \lambda^{i+j}) \cap B(x_k^i,\lambda^i) \neq \emptyset \}, 
\]
$N_j^{i,k}  = \# \mathcal B_j^{i,k}$ for $j\geq 0$,  and 
\[
A_k^i = B(x_k^i, \lambda ^i )  \setminus ( \bigcup _{j=1}^\infty \bigcup_{l \in I_{i+j} } B(x_l ^{i+j}, \lambda^{i +j})).
\]

Now define $u_{i,k}=u \in L^{1,p}(\Omega)$ by
\[
 u(x) = \max\left(0, \min\left(1, 3-\lambda^{-i}\dist(x,x_k^i)\right)\chi_{\Omega_1}(x) \right).
\]
Without loss of generality we can assume that the extension operator applied to any function $0\leq u\leq 1$ also satisfies $0\leq Eu\leq 1$.
We then have
\begin{equation}\label{eq:ugradient}
\|u\|_{L^{1,p}(\Omega)}^p \le \int_{B(x_k^i,3\lambda^i)} |\nabla u(x)|^p\,dx \le (3^n\mm_n(B(0,1)) )\lambda^{i(n-p)} .
\end{equation}

By Proposition \ref{prop:md} and because $\lambda<r$, we have
\begin{equation*}
\min\left(\mm_n(\Omega_1\cap B(x_l^{i+j},\lambda^{i+j}/2\sqrt{n})),\mm_n(\Omega_2\cap B(x_l^{i+j},\lambda^{i+j}/2\sqrt{n}))\right) \ge C_1(n,p)\|E\|^{-n}\dfrac{\lambda^{n(i+j)}}{2^n n^{n/2}}
\end{equation*}
for every $B(x_l^{i+j},\lambda^{i+j}) \in \mathcal {B}_j^{i,k}$. (In the case $n-1<p<n$ Proposition \ref{prop:md} will give a better estimate with $\|E \|^{-p}$ in the above estimate. We shall comment about this case in a remark at the end of the proof.)
Applying Lemma \ref{lma:sobpoinc} where again the removed set $F=\emptyset$, for the corresponding cube $Q(x_l^{i+j}, 2\lambda^{i+j}/\sqrt{n})$ centered at $x_l^{i+j}$ and side length  $2\lambda^{i+j}/\sqrt{n}$ (thus containing the ball $B(x_l^{i+j}, \lambda^{i+j}/\sqrt{n})$ and contained in the ball $B(x_l^{i+j},\lambda^{i+j})$),  we have
\begin{equation}
\begin{split}
\int_{B(x_l^{i+j},\lambda^{i+j})}|\nabla Eu(y)|^p\,\d y
& \ge 
\int_{Q(x_l^{i+j},2\lambda^{i+j}/\sqrt{n}  )}|\nabla Eu(y)|^p\,\d y\\
&\ge \dfrac{C_2(n,p)C_1(n,p)^{(n-p)/n}}{ n^{(n-p)/2}}\|E\|^{p-n}\lambda^{(n-p)(i+j)}.
\end{split} \label{eq:estimate1}
\end{equation}
Thus, since the balls $\{B(x_l^{i+j},\lambda^{i+j})\}_l$ are pairwise disjoint, by summing those ones belonging to $\mathcal {B}_j^{i,k} $ and by using \eqref{eq:estimate1} and \eqref{eq:ugradient}, we get the estimate
\begin{align*}
\dfrac{C_2(n,p)C_1(n,p)^{(n-p)/n}}{ n^{(n-p)/2}} N_j^{i,k} \|E\|^{p-n}\lambda^{(n-p)(i+j)} & \le \sum_{B\in\mathcal {B}_j^{i,k} }\int_{B}|\nabla Eu(y)|^p\,\d y\\
 & \le \int_{\mathbb R^n}|\nabla Eu(y)|^p\,\d y\\
 & \le \|E\|^p \|u\|_{L^{1,p}(\Omega)}^p\\
 & \le (\mm_n(B(0,1)) 3^n) \|E\|^p\lambda^{i(n-p)}.
 \end{align*}
This implies the bound
\begin{equation}
N_j^{i,k} \le \frac{\mm_n(B(0,1)) 3^n  n^{(n-p)/2}}{C_2(n,p)C_1(n,p)^{(n-p)/n}} \|E\|^{n}\lambda^{-j(n-p)} \label{eq:estimate2}
\end{equation}
for every $i,j\in\N$ and $k\in I_i$.

Let us next estimate the $\mathcal H^{n-1}$-measure of the $(n-1)$-projections of the sets
$$F_i=\bigcup _{j=1}^\infty \bigcup_{l \in I_{i+j} } B(x_l ^{i+j}, \lambda^{i +j}) $$
for all $i\geq 1$.
By applying the estimate \eqref{eq:estimate2} and by the choice of $\lambda$, for every $i\geq 1$ and $m=1,\dots,n$,

\begin{align*}
\mathcal{H}^{n-1}(P_m(F_i)) &\leq \sum_{j=1}^\infty N_j^{i,k}(2\lambda^{i+j})^{n-1}\\
&\le 2^{n-1} \left(\frac{\mm_n(B(0,1)) 3^n n^{(n-p)/2}}{C_2(n,p)C_1(n,p)^{(n-p)/n}}\right) \|E\|^{n}\lambda^{i(n-1)} \sum_{j=1}^\infty\lambda^{j(p-1)}\\
& = 2^{n-1}\left(\frac{\mm_n(B(0,1)) 3^n n^{(n-p)/2}}{C_2(n,p)C_1(n,p)^{(n-p)/n}}\right)\|E\|^{n}\dfrac{\lambda^{p-1}}{1-\lambda^{p-1}} \lambda^{i(n-1)} \\ 
& \le \dfrac{C_1(n,p)\|E\|^{-n}}{2 n\cdot 4^n} \left( \dfrac{2\lambda^{i}}{\sqrt{n}}\right)^{n-1}.
\end{align*}

Note that in Proposition \ref{prop:md} one can always assume $C_1(n,p)\|E\|^{-n} < 1$.

Suppose now that $s < \dim_{\mathcal H} (G)$. By Lemma \ref{lma:dim} there exist  $i_0 \in \mathbb{N}$ and $k_0 \in I_{i_0}$ such that $N_j^{i_0,k_0} \geq\lambda^{-js}$ for all $j \geq  0$.

For those  fixed values $i_0,k_0$ and using the above estimate on the $\mathcal{H}^{n-1}$-measure of $P_m(F_i)$, for the case $i=i_0+j$, $j\geq 0$, we can apply Lemma \ref{lma:sobpoinc} to the function $u_{i_0,k_0}=u$, that was defined before. That is,
\begin{equation}
\begin{split}
\int_{A_l^{i_0+j}}|\nabla Eu(y)|^p\,\d y & \ge 
\int_{A_l^{i_0+j}\cap Q(x_l^{i_0+j}, 2\lambda^{i_0+j}/\sqrt{n})}|\nabla Eu(y)|^p\,\d y \\
&=\int_{ Q(x_l^{i_0+j}, 2\lambda^{i_0+j}/\sqrt{n})\setminus F_{i_0+j}}|\nabla Eu(y)|^p\,\d y \\
&\ge C(n,p)\|E\|^{p-n}\lambda^{(n-p)(i_0+j)},\notag 
\end{split}
\end{equation}
where $B(x_l^{i_0+j},\lambda^{i_0+j}) \in \mathcal {B}_j^{i_0,k_0}$. Now,
by \eqref{eq:ugradient}, and by summing over all the scales $j \geq 0$, we get 
\begin{align*}
   C(n)\|E\|^p\lambda^{i_0(n-p)}  \ge \|E\|^p \|u\|_{L^{1,p}(\Omega)}^p
   &\ge \int_{\mathbb R^n}|\nabla Eu(y)|^p\,\d y\\
   & \ge \sum_{j=0}^\infty \sum_{\{l \in I_{i_0+j}:\, B(x^{i_0+j}_{l},\lambda^{i_0+j})\in \mathcal B_{j}^{i_0,k_0}   \}} \int_{A_l^{i_0+j}}|\nabla Eu(y)|^p\,\d y
    \\
    &\ge \sum_{j=0}^\infty N_{j}^{i_0,k_0}C(n,p)\|E\|^{p-n}\lambda^{(n-p)(i_0+j)}\\
   & \ge \sum_{j=0}^\infty \lambda^{-js}C(n,p)\|E\|^{p-n}\lambda^{(n-p)(i_0+j)}\\
   &
    = C(n,p)\|E\|^{p-n}\frac{\lambda^{i_0(n-p)}}{1-\lambda^{n-p-s}}.
\end{align*}
This implies (observe that by the choice of $\lambda$ we have $\lambda\leq  C(n,p)\|E\|^\frac{-2n}{p-1}$)
\[
s \le n - p - \frac{\log\left(1-C(n,p)\|E\|^{-n}\right)}{\log(\lambda)}
  \le n - p - \frac{C(n,p)}{\|E\|^{n}\log(\|E\|)}.
\]
Since $s< \dim_{\mathcal H} (G)$ was chosen arbitrarily, this concludes the proof of (2).
\end{proof}

\begin{remark}
Let us make a remark on the case $ n-1<p<n$. In this case, by applying Proposition \ref{prop:md} $(b)$ we could slightly improve the estimates in the previous theorem. We would have that for the function $u$ defined above,
\[
\int_{A_l^{i+j}}|\nabla Eu(y)|^p\,\d y \ge C(n,p)\|E\|^{-p\left(\frac{n-p}{n}\right)}\lambda^{(n-p)(i+j)},
\] 
and therefore 
\[
s \le n - p - \frac{\log\left(1-C(n,p)\|E\|^{-p\left(\frac{n-p}{n}\right) -p}\right)}{\log(\lambda)}
  \le n - p - \frac{C(n,p)}{\|E\|^{2p-\frac{p^2}{n}}\log(\|E\|)}.
\]
\end{remark}

\section{Two-sided points of $BV$-extension domains}\label{sec_BV}

For a given domain $\Omega\subset\R^n$ the space of functions of bounded variation in $\Omega$ is
$$BV(\Omega)=\{u\in L^1(\Omega):\, \| D u\|(\Omega)<\infty\},$$
where
$$\Vert D u\Vert (\Omega)=\sup\left\{\int_{\Omega}u \,\text{div} (v)\, dx:\, v\in C^{\infty}_{0}(\Omega;\R^n) ,\, |v|\leq 1\right\}$$
denotes the total variation of $u$ on $\Omega$.
We endow this space with the norm $ \Vert u\Vert_{BV(\Omega)}=\Vert u\Vert_{L^1(\Omega)}+\Vert D u\Vert (\Omega)$. We say that $\Omega$ is a $BV$-extension domain if there exists a constant $C>0$ and a (not necessarily linear) extension operator $T\colon BV(\Omega)\to BV(\R^n)$ so that $Tu|_{\Omega}=u$ and 
$$\|Tu\|_{BV(\R^n)}\leq C\|u\|_{BV(\Omega)} $$
for all $u\in BV(\Omega)$ and where $C>0$ is an absolute constant, independent of $u$.

Let us point out that $\Omega$ being a $W^{1,1}$-extension domain always implies that it is also a $BV$-extension domain (see \cite[Lemma 2.4]{KMS2010}).

A Lebesgue measurable subset $E\subset \R^n$ has finite perimeter in $\Omega$ if $\chi_E\in BV(\Omega)$, where $\chi_E$ denotes the characteristic function of the set $E$. We set $P(E,\Omega)=\|D \chi_E\|(\Omega)$ and call it the perimeter of $E$ in $\Omega$. Moreover, the measure theoretic boundary of a set $E \subset \R^n$ is defined as 
$$\partial^M E=\left\lbrace x\in\R^n\,:\, \limsup_{r\searrow 0}\frac{|E\cap B(x,r)|}{|B(x,r)|}>0 \;\text{and}\; 
 \limsup_{r\searrow 0}\frac{|(\R^n\setminus E)\cap B(x,r)|}{|B(x,r)|}>0\right\rbrace,
$$
and for a set of finite perimeter in $\Omega$ one always has $P(E,\Omega)=\mathcal{H}^{n-1}(\partial^M E\cap\Omega)$. Finally, let us recall the useful coarea formula for $BV$ functions. Namely, for a given a function $u\in BV(\Omega)$, the superlevel sets $u_t=\{x\in\Omega:\, u(x)\geq t\}$ have finite perimeter in $\Omega$ for almost every $t\in\R$ and
 \begin{equation}\label{eq:coarea}
     \|Du\|(\Omega)=\int^{\infty}_{-\infty} P(u_t,\Omega)\,dt .
 \end{equation} 
 
\begin{proof}[Proof of Theorem \ref{thm:BV}]

We want to prove that $\dim_{\mathcal{H}}(\mathcal{T}_{\Omega})\leq n-1$.
Similarly to the beginning part of the proof of Theorem \ref{thm:main} and reasoning by contradiction assume that there exists a set $G\subset\partial\Omega\cap B(x_0,r_0)$, with $r\in (0,1)$, $x_0\in G$, and two connected components $\Omega_1,\Omega_2\subset B(x_0,3r_0)\cap \Omega$ for which $G\subset \partial\Omega_1\cap \partial\Omega_2$ such that $\dim_{\mathcal{H}}(G)>n-1$.

Consider the set $E=B(x_0,r_0)\cap \Omega_1$ for which we have $\chi_E\in BV(\Omega)$. Take any measurable function $v$ in $\R^n$ so that $v|_{\Omega}=\chi_E$. 

Note that $\widetilde E_t\cap \Omega= E$ for every $t\in (0,1)$ for the superlevel sets $\widetilde E_t=\{x\in\R^n:\, v(x)\geq t\}$. By using the measure density condition proved in \cite[Proposition 2.3]{GR2021} applied to both connected components $\Omega_1$ and $\Omega_2$, we get that there exists $c>0$ so that
$$\mm_n(\Omega_i\cap B(x,r))\geq c r^n $$
for $i=1,2$ and all $x\in G$, $r\in (0,r_0)$. In particular, for every $x\in G$ we have
\begin{align*}
    &\limsup_{r\searrow 0}\dfrac{\mm_n(B(x,r)\cap \widetilde E_t)}{\mm_n(B(x,r)}\geq \limsup_{r\searrow 0}\dfrac{\mm_n(B(x,r))\cap \Omega_1)}{\mm_n(B(x,r))}>0
\end{align*}
and
\begin{align*}
   & \limsup_{r\searrow 0}\dfrac{\mm_n(B(x,r)\cap (\R^n\setminus \widetilde E_t))}{\mm_n(B(x,r))}\geq \limsup_{r\searrow 0}\dfrac{\mm_n(B(x,r)\cap \Omega_2)}{\mm_n(B(x,r))}>0.
\end{align*}
This means that $G\subset \partial^M \widetilde E_t$. 
Hence, $\mathcal{H}^{n-1}(\partial^M \widetilde E_t)\geq \mathcal{H}^{n-1}(G)=\infty$, so $\widetilde E_t$ does not have finite perimeter in $\R^n$ for any $t\in(0,1)$. Hence, by the coarea formula \eqref{eq:coarea}, $v \notin BV(\R^n)$.

\end{proof}

\section{Examples}\label{sec:examples}

\subsection{Sharpness of the estimate for $p=1$}\label{sec:sharpnessp1}

The following example shows the sharpness of Theorem \ref{thm:main} (1). In this case, when $p=1$, we do not need to care about the norm of the extension operator and consequently, we can rely on previous non-quantitative characterizations of $W^{1,1}$-extension domains.

\begin{example}\label{2dexample}
Let us define 
\[
 \Omega_2 = (-1,1)^2 \setminus \{(x,y)\,:\,|y| \le \dist(x,C), 0\le x \le 1\}
\]
with $C \subset [0,1]$ a Cantor set with $\dim_{\mathcal H }(C) = 1$ and $\mathcal H^1(C) = 0$.
See Figure \ref{fig:cantorp1} for an illustration of the domain $\Omega_2$.

Then $\Omega_2$ is a $W^{1,1}$-extension domain and 
\[
\dim_{\mathcal H}(\mathcal T) = 1.
\]
It is easy to see that $\dim_{\mathcal H}(\mathcal T) = 1$, since
$\mathcal T = (C \times \{0\}) \setminus \{(0,0)\}$. In order to see that $\Omega_2$ is a $W^{1,1}$-extension domain, one can use the following characterization from \cite{KRZ2017} for bounded planar simply-connected domains: $\Omega$ is a $W^{1,1}$-extension domain if and only if
\begin{equation}\label{eq:curvep1}
\begin{split}
 \text{there exists a constant }K \text{ so that for every } x,y \in \Omega^c \text{ there exists a  curve }\\
 \gamma \subset \Omega^c \text{ with } x,y \in \gamma, \ell(\gamma) \le K|x-y|, \text{ and } \mathcal H^1(\gamma \cap \partial \Omega) = 0.
 \end{split}
\end{equation}
Now, the domain  $\Omega_2$ clearly satisfies \eqref{eq:curvep1}  and is thus a $W^{1,1}$-extension domain.
\end{example}

\begin{figure}
    \centering
    \includegraphics[width=0.4\columnwidth]{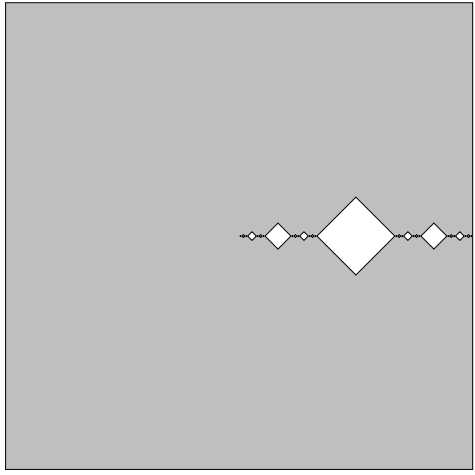}
    \caption{The domain showing the sharpness of Theorem \ref{thm:main} (1). The set $\mathcal T$ here is the fat Cantor set without its left-most point.}
    \label{fig:cantorp1}
\end{figure}

Let us remark that Example \ref{2dexample} can also be generalized to higher dimensions $n > 2$ by defining $\Omega \subset \mathbb R^n$ as a product $\Omega_2 \times  (-1,1)^{n-2}$. It is then clear that
\[
\dim_{\mathcal H}(\mathcal T) = n-1.
\]
The fact that $\Omega$ is a $W^{1,1}$-extension domain does not seem to immediately follow from known explicit results. One way to see that it is a $W^{1,1}$-extension domain is the following. Observe that the proof in \cite{KZ2020} of the fact that a product of $W^{1,p}$-extension domains, with $p>1$, is again a $W^{1,p}$-extension domain relies on the explicit form of the extension operators (which in that case can always be assumed to be a Whitney extension operator). In the case $p=1$ it is unknown if the extension can always be done with a Whitney-type extension. However, the extension operator constructed in \cite{KRZ2017} for simply-connected planar domains, and in particular for $\Omega_2$ is of Whitney-type. Thus, the argument in 
\cite{KZ2020} goes through for our product domain $\Omega$.

\subsection{A bound for the estimates for $p>1$}\label{sec:sharpness}

The case $p>1$ requires more work than the case $p=1$, since the estimate in Theorem \ref{thm:main} depends on the norm of the extension operator.
The assignment of reflected cubes in the construction of the extension operator, and the estimate of the norm of the extension operator follow roughly the proof of the sufficiency of the characterizing curve condition of planar simply connected $W^{1,p}$-extension domains \cite{KRZ2015}.

Let us describe the family of domains $\Omega_\lambda$ we consider, where  $\lambda \in (0,1/2)$ refers to the contraction ratio of the Cantor set $C_\lambda \subset \mathbb R^{n-1}$. The Cantor sets $C_\lambda$ we use are the standard ones obtained as $C_\lambda = \prod_{i=1}^{n-1}K_\lambda$ with $K_\lambda$ being the Cantor set on the unit interval given as the attractor of the iterated function system $\{f_1(x) = \lambda x, f_2(x) = \lambda x + 1- \lambda\}$.

We define first a set
\[
 D = (0,1)^{n-2}\times \left((-2,1) \times (-3/2,3/2) \setminus [-1,0]\times[-1,1]\right)
\]
and then the actual domain by carving out part of $D$:
\[
\Omega_\lambda  = D \setminus N_\lambda,
\]
where
\[
N_\lambda = \left\{(x_1,\dots,x_n) \in [0,1]^n\,:\, |x_n| \le \dist((x_1,\dots,x_{n-1}),C_\lambda)\right\}.
\]
Then, the set of two-sided points for $\Omega_\lambda$ is
\[
 \mathcal{T}_{\Omega_\lambda} = C_\lambda \times \{0\}
\]
and so it has dimension
\begin{equation}\label{eq:Tlambdadim}
\dim_{\mathcal H}(\mathcal{T}_{\Omega_\lambda}) = \dim_{\mathcal H}(C_\lambda) =  -\frac{(n-1)\log 2}{\log \lambda}.
\end{equation}
Our aim is to build an extension operator $E_\lambda$ from $L^{1,p}(\Omega_{\lambda})$ to $L^{1,p}(\R^n)$ for which we have
$$\dim_{\mathcal H} (C_{\lambda})\geq n-p-\frac{C(n,p)}{|| E_{\lambda}||} .$$

It is enough to construct an extension operator $E_\lambda \colon L^{1,p}(\Omega_\lambda) \to L^{1,p}(D)$, since the extension from $L^{1,p}(D)$ to $L^{1,p}(\mathbb R^n)$ is independent of $\lambda$, and exists since $D$ is a Lipschitz domain.  Moreover, our definition of $E_{\lambda}$ will be independent of $p$ and will give a bounded operator between the Sobolev spaces $W^{1,p}(\Omega_\lambda)$ and  $W^{1,p}(D)$.
From now on  consider $\lambda \in (0,1/2)$ fixed and we will denote the extension operator by $E$ instead of $E_{\lambda}$ to simplify notation.

Below by a dyadic cube we mean a set of the form $Q = [0,2^{-k}]^n + \mathtt {j} \subset \R^n$ for some $k \in \mathbb Z$ and $\mathtt{j} \in 2^{-k}{\mathbb Z}^n$.
Let $\mathcal W = \{Q_i\}_{i\in \mathbb N}$ be a Whitney decomposition of the interior of $N_\lambda$ and $\widetilde{\mathcal W} = \{\widetilde Q_i\}_{i \in \mathbb N}$  a Whitney decomposition of $\R^n\setminus N_{\lambda}$. This is 
\begin{itemize}
    \item[(W1)] Each $Q_i$ is a closed dyadic cube inside $N_{\lambda}$.
    \item[(W2)] $N_{\lambda}=\bigcup_i Q_i$ and for every $i \ne j$ we have $\text{int}(Q_i) \cap \text{int}(Q_j) = \emptyset$.
    \item[(W3)]  For every $i$ we have $\sqrt{n}\ell(Q_i) \le \dist(Q_i,\partial N_{\lambda})\le 4\sqrt{n} \ell(Q_i)$,
    \item[(W4)] If $Q_i\cap Q_j \ne \emptyset$, we have $\frac{1}{4}\ell(Q_i) \le  \ell(Q_j)\le 4\ell(Q_i)$. 
\end{itemize}

The definition of $\widetilde{\mathcal W}$ goes parallel. See \cite[Chapter VI]{stein} for the existence of such Whitney decompositions.  Consider also the subfamily of Whitney cubes 
$$\mathcal{V}=\{Q\in \mathcal{W}:\,Q\cap ([0,1]^{n-1}\times \{0\})\neq \emptyset\}.$$

Let us also distinguish an important subset of $\Omega_{\lambda}$, that we call 
$$\widetilde Q_0=  (0,1)^{n-2}\times \left((-2,1) \times (-3/2,3/2) \setminus [-1,1]\times[-1,1]\right).$$
Note that for every $\widetilde Q\in \widetilde{\mathcal{W}}$ we have $\partial ((0,1)^{n-1}\times (-1,1))\cap \text{int}( \widetilde Q)=\emptyset$.

Let $u \in  W^{1,p}(\Omega_\lambda)$ be given and choose $\{\psi_i\}_{i \in \mathbb N}$ a partition of unity subordinate to the open cover $\{ (9/8)\text{int}(Q_i)\}_{i\in\N}$ and so that $|\nabla \psi_i(x)|\lesssim \ell (Q_i)^{-1}$.

We will assign a value 
\[
a_i = \frac1{\mm_n(\widetilde Q_{R(i)})}\int_{\widetilde Q_{R(i)}}u(x)\,dx
\]
for every $i \in \N$, where the function $R \colon \mathbb N \to  \mathbb N$ is defined as follows. If $Q_i\in \mathcal{V}$, then
$R(i) = 0$. If $Q_i \notin \mathcal{V}$ we assign $R(i)$ to be the unique index so that $Q_i$ and $\widetilde Q_{R(i)}$ belong to the same half-space $\{x_n<0\}$ or $\{x_n>0\}$,  $P_n(Q_i) \subset P_n(\widetilde Q_{R(i)})$,  $\ell(\widetilde Q_{R(i)})\leq 2\ell(Q_i) $ where $P_n \colon \mathbb R^n \to \mathbb R^{n-1} \colon (x_1,\dots,x_n) \mapsto (x_1,\dots,x_{n-1})$, and $\widetilde Q_{R(i)}$ is the closest cube to $Q_i$ with the first three properties.

Now, we define the extension of the function $u$ by
\begin{equation}\label{eq:ext.op.example2}
Eu(x) = \begin{cases}
u(x),& \text{if } x \in \Omega_\lambda\\
\sum_{i=1}^\infty a_i\psi_i(x), &\text{  if $x\in \text{int}(N_{\lambda})$, }\\
0,  &\text{  if $x\in \partial N_{\lambda}\cap D$, }.
\end{cases}
\end{equation}

  Let us explain first why $Eu\in L^{1,p}(D)$. On the one hand $Eu\in L^{1,p}(\Omega_\lambda)$ and on the other hand we will see later that $Eu\in L^{1,p}(\text{int}(N_\lambda))$. We now explain how one can get that $Eu\in L^{1,p}(D\setminus C_{\lambda})$ and that will be sufficient since $C_{\lambda}$ is a removable set. For that it is enough to notice that the trace of $u$
$$Tu(x)=\lim_{r\to 0} \dfrac{1}{\mathcal{L}^n(B(x,r)\cap \Omega_\lambda)}\int_{B(x,r)\cap \Omega_\lambda}u(y)\,dy $$
on $\partial N_{\lambda}\setminus C_{\lambda}$  coincides with that of $Eu|_{\text{int}(N_\lambda)}$
$$TEu(x)=\lim_{r\to 0} \dfrac{1}{\mathcal{L}^n(B(x,r)\cap N_\lambda)}\int_{B(x,r)\cap \text{int}(N_\lambda)}Eu(y)\,dy .$$
This follows immediately from the definition of $E$.

 To conclude that $Eu$ is an extension operator it remains to  control the $L^p$-norm of the gradient of the extension on $\text{int}(N_{\lambda})$ by the $L^p$-norm of the gradient of the initial function. 

We know that $\text{supp}(\psi_i)\subseteq \frac{9}{8}Q_i$ and that $|\nabla \psi_i(x)|\lesssim \ell (Q_i)^{-1}$ for every $x$ and $i\in\N$, so it is clear that for every $x\in Q_i$,
\[
|\nabla Eu(x)| \le \left|\sum_{Q_j\cap Q_i\neq \emptyset} \nabla \psi_j(x)(a_j-a_i)\right| \lesssim \sum_{Q_j\cap Q_i\neq \emptyset} \ell (Q_j)^{-1}|a_j-a_i|
\]
Now, if we take a cube $Q_i\in \mathcal{W}$, using that at most $C(n)$ other cubes of the Whitney decomposition are intersecting it and that $\ell(Q_i)\sim \ell(Q_j)$ if $Q_i\cap Q_j\neq \emptyset$ we write
\begin{equation}\label{eq:norm in a cube}
\begin{split}
\Vert \nabla Eu\Vert^{p}_{L^p(Q_i)}&=\int_{Q_i} |\nabla E u(x)|^p\,dx\lesssim \int_{Q_i} \sum_{Q_j\cap Q_i\neq \emptyset}\ell (Q_j)^{-p}|a_i-a_j|^p\,dx \\
&\lesssim \ell(Q_i)^{n-p}\sum_{Q_j\cap Q_i\neq \emptyset}|a_i-a_j|^p.
\end{split}
\end{equation}
 It will be useful to work with chains of Whitney cubes that we next define.
Given $i,j$ so that $Q_i\cap Q_j\neq \emptyset$ and $Q_i,Q_j\notin\mathcal V$ we define the chain of cubes joining $ \widetilde Q_{R(i)}$ with $\widetilde Q_{R(j)} $, and denote it by  $C(\widetilde Q_{R(i)}, \widetilde Q_{R(j)})$, to be a minimal family of Whitney cubes whose union's interior  is a connected set containing both the interiors of $ \widetilde Q_{R(i)}$ and $\widetilde Q_{R(j)} $. Note that we always have $\# C(\widetilde Q_{R(i)}, \widetilde Q'_{R(j)})\leq C_0(n) $. Suppose $Q_i\notin\mathcal{V}$ is a cube so that there exists $Q_j\in\mathcal V$ with $Q_i\cap Q_j\neq\emptyset$. For the associated cube $\widetilde Q_{R(i)}$   we define $C(\widetilde Q_{R(i)}, \widetilde Q_0)$ as a minimal family of sets in $\widetilde W\cup \{\widetilde Q_0\}
$ whose union's interior  is a connected set containing both the interiors of $ \widetilde Q_{R(i)}$ and $\widetilde Q_0 $ and so  that every $\widetilde Q\in C(\widetilde Q_{R(i)},\widetilde Q_0)  $ satisfies $P_n(\widetilde Q_{R(i)})\subset P_n(\widetilde Q)$. 

We can assume there is an order in the chain when moving from $ \widetilde Q_{R(i)}$ to $\widetilde Q_{R(j)} $ and call $\widetilde Q_{\text{next}}$ the next cube in the chain after $\widetilde Q$. We write 
$$C(\widetilde Q_{R(i)},\widetilde Q_{R(j)})=\{\widetilde Q_{R(i)}, (\widetilde {Q}_{R(i)})_{\text{next}},\dots, \widetilde{Q}_{R(j)}\}.$$
To ease notation in the following sums from now on we write $$C_{i,j}=C(\widetilde Q_{R(i)}, \widetilde Q_{R(j)}) \setminus \{\widetilde Q_{R(j)}\}\quad\text{and}\quad C_{i,0}=C(\widetilde Q_{R(i)}, \widetilde Q_0)\setminus \{\widetilde Q_0\}.$$
Note that if $\widetilde{Q}_i \not \in \mathcal V$ and there does not exists $\widetilde{Q}_j$ such that $\widetilde{Q}_j \in \mathcal V$ and $\widetilde{Q}_i \cap \widetilde{Q}_j \neq \emptyset$ we define $C_{i,0} = \emptyset$.

Let also write $$\mathcal I=\left\{\widetilde Q\in \widetilde W:\, \widetilde Q=\widetilde Q_{R(i)} \;\; \text{for some}\; i\geq 1\right\}. $$

 We assert that the following claim holds.

\begin{claim}\label{claim_example} With the above notation and for every $r>0$ we have the following.
\begin{enumerate}
    \item For every $Q_i\notin \mathcal{V}$ 
     \begin{align*}
 \Vert \nabla  Eu\Vert^{p}_{L^p(Q_i)}\lesssim&  \sum_{ \{ \widetilde Q\in \mathcal I\,:\, \# C(\widetilde Q_{R(i)},\widetilde Q)\leq C_0(n)\} } \int_{\widetilde Q}|\nabla u(x)|^p\,dx\\
        &+\ell(Q_i)^{n-p-rp} D(r,p)\sum_{ \widetilde Q\in C_{i,0}} \ell(\widetilde Q)^{p-n+rp}\int_{\widetilde Q\cup\widetilde{Q}_{\text{next}}} |\nabla u(x)|^p\,dx,
       \end{align*}  
        where $D(r,p)=(1-2^{\frac{-rp}{p-1}})^{1-p}$, and for every $Q_i\in \mathcal{V}$, we have
  \begin{align*}
 \Vert \nabla Eu\Vert^{p}_{L^p(Q_i)}
 \lesssim&\, \ell(Q_i)^{n-p-rp}\sum_{\substack{Q_j\cap Q_i\neq \emptyset\\Q_j\notin\mathcal{V}}} D(r,p) \sum_{ \widetilde Q\in C_{j,0}} \ell(\widetilde Q)^{p-n+rp}\int_{\widetilde Q\cup\widetilde{Q}_{\text{next}}} |\nabla u(x)|^p\,dx.
 \end{align*} 
 \item For a given $\widetilde{Q}\in\widetilde{\mathcal{W}}$  and  $k\in\Z$ we have $$\#{\left\{Q_i\in \mathcal W \setminus\mathcal V:\,Q_i\;\text{has a neighbouring cube in}\; \mathcal{V},\;\ell(\widetilde{Q})=2^k\ell (Q_i),\; \widetilde{Q}\in C_{i,0}\right\}}\lesssim 2^{-(n-1)\frac{k\log 2}{\log \lambda}}. $$
\end{enumerate}

\end{claim}

Assuming for a moment that the claim is true let us show how one can estimate the full norm $ \Vert \nabla Eu\Vert^{p}_{L^p(N_{\lambda})}$.  We first use Claim \ref{claim_example} (i) and change the order of summation to get

 \begin{align*}
 \Vert \nabla Eu\Vert^{p}_{L^p(N_{\lambda})}=&\,\sum_{Q_i\in \mathcal W} \Vert \nabla E u\Vert ^{p}_{L^p(Q_i)}=\sum_{Q_i\notin \mathcal V} \Vert \nabla E u\Vert ^{p}_{L^p(Q_i)}+\sum_{Q_i\in \mathcal V} \Vert \nabla E u\Vert ^{p}_{L^p(Q_i)} \\
 \lesssim&\,  \sum_{\widetilde{Q}\in \mathcal{I}}\sum_{ \{ i:\, \# C(\widetilde Q_{R(i)},  \widetilde Q)\leq C_0(n) \}}\int_{\widetilde Q}|\nabla u(x)|^p\,dx  \\
 &+ 2\sum_{\widetilde{Q}\in\widetilde{\mathcal{W}}}\sum_{\widetilde{Q}\in C_{i,0}}
 \ell(Q_i)^{n-p-rp} D(r,p) \ell(\widetilde Q)^{p-n+rp}\int_{\widetilde Q\cup\widetilde{Q}_{\text{next}}} |\nabla u(x)|^p\,dx \\
 \lesssim &\, \Vert \nabla u \Vert^{p}_{L^p(\Omega_\lambda)}+ \sum_{\widetilde{Q}\in\widetilde{\mathcal{W}}}\sum_{\widetilde{Q}\in C_{i,0}}
 \ell(Q_i)^{n-p-rp} D(r,p) \ell(\widetilde Q)^{p-n+rp}\int_{\widetilde Q\cup\widetilde{Q}_{\text{next}}} |\nabla u(x)|^p\,dx.
\end{align*}

Moreover, by Claim \ref{claim_example} (ii) it follows that
$$\sum_{\{i:\,\widetilde{Q}\in C_{i,0}\}}  \ell(Q_i)^{n-p-rp}\lesssim \sum^{\infty}_{k=0}2^{-(n-1)\frac{k\log 2}{\log \lambda}} (2^{-k}\ell(\widetilde{Q}))^{n-p-rp}=\frac{\ell(\widetilde Q)^{n-p-rp}}{1-2^{-n+p-(n-1)\frac{\log 2}{\log\lambda}+rp}}.$$
So, joining these facts together we get
\begin{align*}
 \Vert \nabla Eu\Vert^{p}_{L^p(N_{\lambda})}&\lesssim \Vert \nabla u \Vert^{p}_{L^p(\Omega_\lambda)}+ \sum_{\widetilde{Q}\in\widetilde{\mathcal{W}}}D(r,p)\left( \frac{1}{1-2^{-n+p-(n-1)\frac{\log 2}{\log\lambda}+rp}}\right)\int_{\widetilde Q\cup\widetilde{Q}_{\text{next}}} |\nabla u(x)|^p\,dx  \\
 &\lesssim \left(\frac{1}{1-2^{\frac{-rp}{p-1}}}\right)^{p-1}\left( \frac{1}{1-2^{-n+p-(n-1)\frac{\log 2}{\log\lambda}+rp}}\right)\Vert  \nabla u\Vert^{p}_{L^p(\Omega_{\lambda})}.
 \end{align*}
Choosing $r=\frac{p-1}{p^2}(n-p-\dim_{\mathcal H} (C_{\lambda}))$, we conclude that
$$|| E||\lesssim \frac{1}{1-2^{\frac{1}{p}(-n+p+\dim_{\mathcal H} (C_\lambda))}},$$
which yields
$$\dim_{\mathcal H} (C_{\lambda})\geq n-p-\frac{C(n,p)}{|| E||}.$$
Let us now prove the Claim \ref{claim_example}.

\begin{proof}[Proof of Claim \ref{claim_example}]
To prove (i) we need to estimate $|a_i-a_j|^p$ in the expression \eqref{eq:norm in a cube}. First note that from \eqref{eq:norm in a cube} one gets 
 \begin{equation}
 \begin{split}
  |a_i-a_j|^p&\leq \left(\sum_{ \widetilde Q\in C_{i,j}}  \left|\frac{1}{\mm_n(\widetilde Q)}\int_{\widetilde Q} u(x)\,dx-\frac{1}{\mm_n(\widetilde{Q}_{\text{next}})}\int_{\widetilde{Q}_{\text{next}}} u(x)\,dx\right|\right)^p\\
        &\lesssim  \left(\sum_{ \widetilde Q\in C_{i,j}} \ell(\widetilde Q)^{1-n}\int_{\widetilde Q\cup\widetilde{Q}_{\text{next}}} |\nabla u(x)|\,dx\right)^p, \label{eq:exa1}
\end{split}
 \end{equation}
    where we are using the Poincar\'e inequality in the last line (see \cite[Lemma 2.2]{jo1981} and also \cite{boj}). Observe that if $Q_j,Q_i\in \mathcal V$ then $R(i)=R(j)=0$ and $|a_j-a_i|=0$. We now consider two cases.
\begin{enumerate}
\item Suppose $i,j$ are so that $Q_i\cap Q_j\neq \emptyset$ and $Q_i,Q_j\notin\mathcal V$.  Then  using \eqref{eq:exa1}, that  $\# C(\widetilde Q_{R(i)}, \widetilde Q_{(R(j)})\leq C_0(n)$,      that the sides of the cubes of the chain  are comparable to that of $\widetilde Q_{R(i)}$, and hence that of $Q_i$, and applying H\"older inequality
\begin{equation}\label{eq:case_1_last_example}
\begin{split}
     |a_i-a_j|^p&\lesssim \sum_{ \widetilde Q\in C_{i,j} } \ell(\widetilde Q)^{(1-n)p}\left(\int_{\widetilde Q\cup\widetilde{Q}_{\text{next}}} |\nabla u(x)|\,dx\right)^p \\
    &\lesssim \ell(Q_i)^{p-n} \sum_{ \widetilde Q\in C_{i,j} } \int_{\widetilde Q\cup\widetilde{Q}_{\text{next}}} |\nabla u(x)|^p\,dx.
    \end{split}
    \end{equation}

\item Suppose $i,j$ are such that $Q_i\cap Q_j\neq \emptyset$, $Q_j\in \mathcal V$ (then $R(j)=0)$ and  $Q_i\notin \mathcal V$ then we fix $r>0$ to be determined later and apply H\"older inequality  to  \eqref{eq:exa1} to get
    \begin{equation}\label{eq:case_2_last_example}
    \begin{split}
        |a_i-a_j|^p&\lesssim \left( \sum_{ \widetilde Q\in C_{i,0}} \ell(\widetilde Q)^{-r}  \ell(\widetilde Q)^{1-n+r}\int_{\widetilde Q\cup\widetilde{Q}_{\text{next}}} |\nabla u(x)|\,dx\right)^p  \\
        &\le \left(\sum_{ \widetilde Q\in C_{i,0}} \ell(\widetilde Q)^{-r\frac{p}{p-1}} \right)^{p-1}  \left(\sum_{ \widetilde Q\in C_{i,0}} \ell(\widetilde Q)^{(1-n+r)p}\left(\int_{\widetilde Q\cup\widetilde{Q}_{\text{next}}} |\nabla u(x)|\,dx\right)^p\right) \\
        & \lesssim \left(\sum^{\infty}_{k=0} (2^k\ell(\widetilde{Q}_{R(i)}))^{-r\frac{p}{p-1}} \right)^{p-1} 
        \left(\sum_{ \widetilde Q\in C_{i,0}} \ell(\widetilde Q)^{p-n+rp}\int_{\widetilde Q\cup\widetilde{Q}_{\text{next}}} |\nabla u(x)|^p\,dx\right) \\
         & \lesssim D(r,p) \ell (\widetilde{Q}_{R(i)})^{-rp}
        \left(\sum_{ \widetilde Q\in C_{i,0}} \ell(\widetilde Q)^{p-n+rp}\int_{\widetilde Q\cup\widetilde{Q}_{\text{next}}} |\nabla u(x)|^p\,dx\right).
        \end{split}
    \end{equation}
 
\end{enumerate}
 Going back to equation \eqref{eq:norm in a cube} for any 
 $Q_i\notin \mathcal{V}$, using \eqref{eq:case_1_last_example} and \eqref{eq:case_2_last_example} we have 
 \begin{align*}
 \Vert \nabla  Eu\Vert^{p}_{L^p(Q_i)}\lesssim&  \,\ell(Q_i)^{n-p}\left(\sum_{Q_j\cap Q_i\neq \emptyset,\,Q_j\notin\mathcal{V}}|a_i-a_j|^p+ \sum_{Q_j\cap Q_i\neq \emptyset,\,Q_j\in\mathcal{V}}|a_i-a_j|^p\right) \\
 \lesssim& \,\ell(Q_i)^{n-p}\left( \sum_{Q_j\cap Q_i\neq \emptyset,\,Q_j\notin\mathcal{V}} \ell(Q_i)^{p-n}\sum_{\widetilde Q\in C_{i,j}} \int_{\widetilde Q\cup\widetilde{Q}_{\text{next}}} |\nabla u(x)|^p\,dx \right.\\
 &+ \left. \sum_{Q_j\cap Q_i\neq \emptyset,\,Q_j\in\mathcal{V}}
 D(r,p) \ell ({Q}_i)^{-rp}\sum_{\widetilde Q\in C_{i,0}} \ell(\widetilde Q)^{p-n+rp}\int_{\widetilde Q\cup\widetilde{Q}_{\text{next}}} |\nabla u(x)|^p\,dx\right) \\
        \lesssim&  \sum_{ \{ \widetilde Q\in\mathcal{I}\,:\, \# C(\widetilde Q_{R(i)},\widetilde Q)\leq C_0(n)\} } \int_{\widetilde Q}|\nabla u(x)|^p\,dx\\
        &+\ell(Q_i)^{n-p-rp} D(r,p)\sum_{ \widetilde Q\in C_{i,0}} \ell(\widetilde Q)^{p-n+rp}\int_{\widetilde Q\cup\widetilde{Q}_{\text{next}}} |\nabla u(x)|^p\,dx,
 \end{align*}
 and if $Q_i\in \mathcal{V}$, using only \eqref{eq:case_2_last_example}
\begin{align*}
 \Vert \nabla Eu\Vert^{p}_{L^p(Q_i)}\lesssim&  \ell(Q_i)^{n-p}\sum_{Q_j\cap Q_i\neq \emptyset,\,Q_j\notin\mathcal{V}}|a_i-a_j|^p  \\
 \lesssim&\, \ell(Q_i)^{n-p-rp}\sum_{Q_j\cap Q_i\neq \emptyset,\,Q_j\notin\mathcal{V}} D(r,p) \sum_{ \widetilde Q\in C_{j,0}} \ell(\widetilde Q)^{p-n+rp}\int_{\widetilde Q\cup\widetilde{Q}_{\text{next}}} |\nabla u(x)|^p\,dx.
 \end{align*}
 which proves (i).

Let us prove (ii). Let us write the Cantor set $C_\lambda\subset[0,1]^{n-1}$ as
$$C_{\lambda}= \bigcap_{i=0}^{\infty}C^{i}_{\lambda}= \bigcap_{i=0}^{\infty}\bigcup_{1\leq j \leq 2^{(n-1)i}} I_{i,j},$$ where $I_{i,j}$ is a translated copy of $[0,\lambda^i]^{n-1}$ for all $i =0,1,2,\ldots$ and $j=1,2,\ldots , 2^{(n-1)i}$.
It is clear that for $i<i'$, any cube   $I_{i,j} $ contains $2^{(n-1)(i'-i)}$ cubes of side length $\lambda^{i'}$.  

 Fix $\widetilde{Q} \in \widetilde{\mathcal{W}}$ and $k \in \N$. Let $t\in \N$ such that $\ell(\widetilde{Q})=2^{-t}$. We count the cardinality of $$A={\left\{Q_i\in \mathcal W \setminus\mathcal V:\,Q_i\;\text{has a neighbouring cube in}\; \mathcal{V},\;\ell(\widetilde{Q})=2^k\ell (Q_i),\; \widetilde{Q}\in C_{i,0}\right\}}.$$
Define $B=\{P_n(Q_i)\}_{Q_i\in A}$,  where $P_n \colon \mathbb R^n \to \mathbb R^{n-1} \colon (x_1,\dots,x_n) \mapsto (x_1,\dots,x_{n-1})$.

 Let $m$ be the least positive integer such that $\lambda^m < 2^{-t}$ and let $l$ be the least positive integer so that  $\lambda^l\leq 2^{-t-k} $. By the properties of the Whitney decomposition, the construction of the Cantor set and the minimality of $m$ it is enough to count $ \#\{Q\in B:\, \dist(Q, I_{m,j})\le C(n)\ell(Q)\} $ for a fixed $I_{m,j}$. Moreover by the selection of $l$  none of the cubes $I_{l,j}$ contains any $Q\in B$.

Because $\lambda^l\leq \ell(Q) $, we have
$$\#\left\{Q\in B:\, \dist(Q,I_{l,j'})\le C(n) \ell(Q)\right\}\leq c(n)$$ 
 for all $I_{l,j'}\subset I_{m,j} $. Finally since $I_{m,j} \cap C^{l}_{\lambda}$ is a disjoint union of $2^{(n-1)(l-m)}$ cubes $I_{l,j'}$ of side length $\lambda^l$ we conclude that
$$ \# A\lesssim\# B \leq c(n) 2^{(n-1)(l-m)}   \lesssim 2^{-k(n-1) \frac{\log 2}{\log \lambda} }. $$

\end{proof}

\end{document}